\documentclass[12pt,leqno]{amsart}
\usepackage{amssymb}
\usepackage{amsthm}
\usepackage{tikz}

\newtheorem{theorem}{Theorem}
\newtheorem{question}{Question}

\newtheorem{lemma}{Lemma}
\newtheorem{proposition}{Proposition}
\newtheorem{corollary}{Corollary}

\def\bZ{\mathbb Z}
\def\bF{\mathbb F}
\def\bQ{\mathbb Q}
\def\bC{\mathbb C}
\def\bG{\mathbb G}

\def\fl{\mathrm{fl}}

\DeclareMathOperator{\Pic}{Pic}
\DeclareMathOperator{\Tr}{Tr}
\DeclareMathOperator{\Nm}{Nm}
\DeclareMathOperator{\Div}{Div}
\DeclareMathOperator{\Gal}{Gal}
\DeclareMathOperator{\End}{End}
\DeclareMathOperator{\Spec}{Spec}
\DeclareMathOperator{\Cl}{Cl}

\begin{document}

\title{On a problem \`a la Kummer-Vandiver for function fields}

\author{Bruno Angl\`es \and Lenny Taelman}

\address{
Universit\'e de Caen, 
CNRS UMR 6139, 
Campus II, Boulevard Mar\'echal Juin, 
B.P. 5186, 
14032 Caen Cedex, France.
}
\email{bruno.angles@unicaen.fr}

\address{ 
Mathematisch Instituut,
Universiteit Leiden,
P.O. Box 9512,
2300~RA Leiden, The Netherlands.
}
\email{lenny@math.leidenuniv.nl}

\maketitle

\begin{abstract}
We use Artin-Schreier base change to construct counterexamples to a Kummer-Vandiver type question for function fields.
\end{abstract}


\section{Introduction}

Let $p$ be a prime number and let $F$ be the maximal real subfield of $\bQ(\mu_p).$ The famous Kummer-Vandiver conjecture asserts that 
\[
	\bZ_p \otimes_\bZ \Pic O_F =\{ 0\}.
\]
It has been verified for all $p$ less than $163,577,856$ \cite{BH}. However, heuristic arguments of Washington suggest that the number of counterexamples $p$ up to $X$ should grow as $\log \log X$, making it difficult to find either counterexamples or convincing numerical evidence towards the conjecture. 

The second author has recently  proven a function field analogue of the Herbrand-Ribet theorem, and formulated a version of the Kummer-Vandiver conjecture in this context \cite{TAE1}. In this note, which complements \cite{TAE1}, we construct counterexamples to this Kummer-Vandiver statement. 

Let us now recall the statement of this analogue. Let $A=\bF_q[T]$ and let $C$ be the Carlitz module over $A$. This is the $A$-module scheme over $\Spec A$ whose underlying group scheme is the additive group $\bG_{a,A}$, and on which
$A$ acts via the $\bF_q$-algebra homomorphism
\[
	\phi\colon A \to \End \bG_{a,A} \colon T \mapsto T + \tau
\]
where $\tau\colon \bG_{a,A} \to \bG_{a,A}$ is the $q$-th power Frobenius endomorphism. It is an example of a Drinfeld module, and in many ways it is an analogue of the multiplicative group in characteristic zero \cite{GOS}.

Let $P\in A$ be monic irreducible, let $k=\bF_q(T)$ be the fraction field of $A$ and let $K/k$ be the extension obtained by adjoining the $P$-torsion points of $C$. Then $K/k$ is Galois and
there is a canonical isomorphism
\[
	\omega_P\colon \Gal(K/k)\to (A/P)^\times,
\]
which one can think of as the mod $P$ Teichm\"uller character. Let $R$ be the integral closure of $A$ in $K$ and put $Y = Y_P = \Spec R$.

Let $C[P]$ be the $P$-torsion subscheme of $C$ and let $C[P]^{\mathrm D}$ be its Cartier dual. Consider the flat cohomology group
\[
	H_P := \mathrm H^1( Y_{P,\fl}, C[P]^{\mathrm D} ).
\] 
This is an $A/P$-vector space on which the Galois group $\Gal(K/k)$ acts, so it decomposes in isotypical components as
\[
	H_P = \bigoplus_{n=1}^{q^{\deg P} - 1} H_P (\omega_P^n).
\]

In \cite{TAE1}, it is shown that for $n$ in the range $1 \leq n< q^{\deg P}-1$ which are
divisible by $q-1$ one has
\[
	H_P(\omega_P^{n-1})\not = \{0\} \quad\text{if and only if}\quad B(n)\equiv 0\pmod{P},
\]
where $B(n)\in k$ is the $n$-th Bernoulli-Carlitz number. The Kummer-Vandiver problem can be stated as follows (see \cite[Question 1]{TAE1}): 

\begin{question}\label{QKV}
Is  $H_P(\omega_P^{n-1}) = \{0\}$ for $n$ not divisible by $q-1$?
\end{question}

The analogy with the classical Kummer-Vandiver conjecture is (implicitly) explained in \cite[Remark 2]{TAE1}: using the Kummer sequence and flat duality it is shown that the classical Kummer-Vandiver conjecture is equivalent
with the statement that
\[
	\mathrm H^1( (\Spec \bZ[\zeta_\ell])_\fl, \mu_\ell^{\mathrm D} ) ( \chi_\ell^{n-1} ) \overset{?}{=} 0 \text{ if $n$ is odd},
\]
where $\ell$ is an odd prime, $\zeta_\ell$ a primitive $\ell$-th root of unity and $\chi_\ell$ denotes the mod $\ell$ cyclotomic character.

In this paper we use Artin-Schreier change of variables and computer calculations to construct counterexamples to the above statement. For example, we use properties of the prime
\[
	P = T^3 - T^2 + 1 \in\bF_3[T]
\]
to show that the prime
\[
	Q = P(T^3-T) = T^9 - T^6 - T^4 - T^3 - T^2 + 1
\]
satisfies  $H_Q(\omega_Q^{9840}) \neq 0$. Note that $9840=n-1$ with $n=(q^{\deg Q}-1)/2$.

The degree of the prime $Q$ is too high to allow for a direct computation of $H_Q$.

In a forthcoming paper we compare the flat cohomology groups of \cite{TAE1} with the group of ``units modulo circular units'' introduced by Anderson \cite{AND}, and show amongst other things that the Kummer-Vandiver problem of \cite{TAE1} is equivalent with Anderson's Kummer-Vandiver conjecture \cite[\S 4.12]{AND}. In particular, the present counterexamples will also serve as counterexamples to Anderson's conjecture.

\subsection*{Acknowledgements} The authors thank the referee for several suggestions and corrections that helped improve the paper.

 \section{Notation}
 
\subsection{$L$-functions}
Let $F/E$ be a finite abelian extension of function fields of curves over $\bF_q$. Assume that $\bF_q$ is algebraically closed in both $E$ and $F$. Let $\chi\colon \Gal(F/E) \to \bC^\times$ be a homomorphism, and let $E_\chi \subset F$ be the fixed field of $\ker \chi$. We set
\[
	L(X,E,\chi)=
	\prod_{v\, {\rm place}\, {\rm of}\, E}
	(1-\chi(v)X^{{\rm deg}v})^{-1} \in \bC[[X]],
\]
 where $\chi(v)=\chi( (v, E_\chi/E) )$ if $v$ is unramified in $E_\chi/E$ and $\chi(v)=0$ otherwise. Here $(-,E_\chi/E)$ denotes the global reciprocity map.
 Recall that $L(X,E,\chi)$ is a rational function and that if $\chi \neq 1$ then
$L(X,E,\chi)$ is a polynomial whose coefficients are algebraic integers.

\subsection{The cyclotomic function fields}
 Let $p$ be a prime number. Let $\bF_q$ be a finite field having $q$ elements, $q=p^s,$ where $p$ is the characteristic of $\bF_q.$ Let $A=\bF_q[T]$ be the polynomial ring in one variable $T$ and let  $k=\bF_q(T)$ be its field of fractions. We denote the set of monic elements in $A$ by $A_+.$ For $n\geq 0,$ we denote the set of elements in $A_+$ of degree $n$ by $A_n.$ We fix ${\overline  k}$, an algebraic closure of $k.$ All finite extensions of $k$ considered in this note are assumed to be contained in ${\overline k}.$ We denote the unique place of $k$ which is a pole of $T$ by $\infty.$
 
 Let $P\in A$  be monic irreducible of degree $d.$ We denote the $P$-th cyclotomic function field by $K_P$ (see \cite{GOS}, chapter 7). Recall that $K_P/k$ is the maximal abelian extension of $k$ such that:
\begin{enumerate}
 \item $K_P/k$ is unramified outside $P$ and $\infty$,
 \item $K_P/k$ is tamely ramified at $P$ and $\infty$,
 \item for every place $v$ of $K_P$ above $\infty,$ the completion of $K_P$ at $v$ is isomorphic to $\mathbb F_q((\frac{1}{T}))({}^{q-1}\sqrt{-T})$.
\end{enumerate}
The Galois group $\Gal(K_P/k)$ is canonically isomorphic with $(A/PA)^\times$ and the subgroup $\bF_q^\times \subset (A/PA)^\times$ is both the inertia and the decomposition group of $\infty$ in $K_P/k$.


 \section{Cyclicity of divisor class groups}\label{seccyc}
 
 \subsection{An Artin-Schreier extension and the function $\gamma$.}
Let $i\colon A_+\rightarrow \mathbb \bZ/p\bZ$ be the function that maps a polynomial
\[
	T^n + \alpha_1 T^{n-1} + \cdots + \alpha_n
\]
to $\Tr_{\bF_q/\bF_p} \alpha_1 \in \bZ/p\bZ$. 
Observe that for all $a,b\in A_{+}$ we have $i(ab)=i(a)+i(b)$.
 
 Let $\theta \in \bar{k}$ be a root of $X^p-X=T$. Then the extension $\tilde{k}$ obtained by adjoining $\theta$ to $k$ is rational and we have $\tilde{k}=\bF_q(\theta)$. The extension ramifies only at $\infty$. The integral closure of $A$ in $\tilde{k}$, which we denote by
$\tilde{A}$, is the polynomial ring $\bF_q[\theta]$ in $\theta$. 

We have an isomorphism of groups $\bZ/p\bZ \to \Gal(\tilde{k}/k)$ given by 
\[
	n \mapsto \sigma_n = \left[ \theta \mapsto \theta - n \right].
\]
Let $( - , \widetilde {k}/k)$ be the Artin symbol for ideals, then 
for all $a\in A_+$ we have \cite[Lemma 2.1]{ANG}
\begin{equation}\label{eqartinsymbol}
	(aA,\widetilde{k}/k)= \sigma_{i(a)}.
\end{equation}

Let $n$ be a positive integer. By \cite[Lemma 3.2]{ANG}  for all $m$ sufficiently large we have
\[
	\sum_{a\in A_m} i(a)a^n = 0.
\]
We define 
\[
	\gamma (n):=\sum_{m\geq0}\sum_{a\in A_m}i(a)a^n\, \in A.
\]

\subsection{Cyclotomic extensions of $k$ and of $\tilde{k}$}

Now, we fix a prime $P$ in $A$ of degree $d$ such that $i(P)\not =0.$ Set $Q(\theta):=P(T)\in \widetilde {A}$. Note that by (\ref{eqartinsymbol}) the polynomial $Q(\theta)\in \bF_q[\theta]$ is irreducible. Its degree is $pd$. 

Let $K_P$ be the $P$-th cyclotomic function field for $A$, with Galois group $\Delta = (A/PA)^\times$, and let $\widetilde{K_Q}$ be the $Q$-th cyclotomic function field for $\widetilde {A}$,
with Galois group $\widetilde{\Delta} = (\widetilde{A}/Q\widetilde{A})^\times$. By \cite[Lemma 2.2]{ANG} we have:
\[
	K_P\subset \widetilde {K_Q}.
\]
Let $L$ be the compositum of $\widetilde{k}$ and $K_P$ inside $\widetilde{K_Q}$. Then $L$ is an abelian extension of $k$ with Galois group $\bZ/p\bZ \times \Delta$.

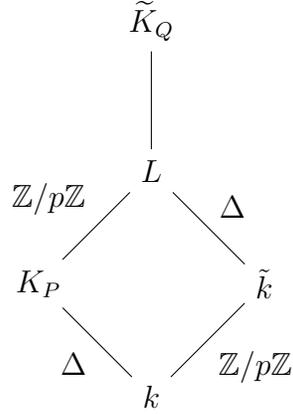
\[
\begin{tikzpicture}[node distance=2cm, auto]
  \node (Ktilde) {$\widetilde{K}_Q$};
  \node (L) [below of=Ktilde] {$L$};
  \node (K) [below of=L, left of=L, node distance=1.5cm] {$K_P$};
  \node (ktilde) [below of=L, right of=L, node distance=1.5cm] {$\tilde{k}$};
  \node (k) [below of=K, right of=K, node distance=1.5cm] {$k$};
  \draw (k) to node {$\Delta$} (K);
  \draw (k) to node [swap] {$\bZ/p\bZ$} (ktilde);
  \draw (K) to node {$\bZ/p\bZ$} (L);
  \draw (ktilde) to node [swap] {$\Delta$} (L);
  \draw (L) to node {} (Ktilde);
\end{tikzpicture}
\]

The inclusion $A/PA \subset \widetilde{A}/Q\widetilde{A}$ induces an injective homomorphism $\Delta\to\widetilde{\Delta}$. On the other hand, we can identify the Galois group of $L$ over $\tilde{k}$ with $\Delta$, and obtain a surjective map
\[
	\widetilde{\Delta} \to \Delta,
\]
which is explicitly given by
\[
	(\widetilde{A}/Q\widetilde{A})^\times \to (A/PA)^\times\colon
	a \mapsto \Nm_{\tilde{k}/k} a.
\]

\subsection{Comparison of $L$-functions}

Let $W_0$ be the ring of Witt vectors of $A/QA$ and let $W=W_0[\zeta_p]$, where $\zeta_p$ is a primitive $p$-th root of unity. Let
$\omega_P\colon \Delta \to W^\times$ and $\omega_Q\colon \widetilde{\Delta}\to W^\times$ be the Teichm\"uller characters. We will denote by $\widetilde{\omega_P}$ the same character as $\omega_P$, but seen as a character on $\Gal(L/\widetilde{k})$. In particular, we have
\[
	\widetilde{\omega_P}=\omega_Q^{\frac {q^{pd}-1}{q^d-1}}.
\]

Let $n$ be an integer such that $1\leq n\leq q^d-2.$ Then \cite[Lemma 2.4]{ANG}
\begin{equation}\label{Ldecomp}
	L(X, L/\tilde{k}, \widetilde {\omega_P}^n)= \prod_{\phi} L(X, L/k, \phi \omega_P^n),
\end{equation}
where $\phi$ runs over all characters of $\Gal(\tilde{k}/k)=\bZ/p\bZ$.

Observe that, if $\phi \not =1,$ then $L(X,\phi \omega_P^n)$ is a polynomial of degree $d$ (apply \cite{ANG}, Lemma 2.3 for both $A$ and $\widetilde A$). Furthermore, we have:
\begin{equation}\label{psiomega}
	L(X, L/k, \psi \omega_P^n)=\sum_{m=0}^{d} 
	\left( \sum_{a\in A_m} \zeta_p^{i(a)} \omega_P(a)^n \right) X^m,
\end{equation}
where $\psi\colon \bZ/p\bZ \to W^\times$ is the character that maps $1$ to $\zeta_p$.

\subsection{Congruences}
Assume that $n$ is not divisible by $q-1$. Then the Bernoulli-Goss polynomial $\beta(n)$ is defined as follows:
\[
	\beta(n) = \sum_{m\geq 0} \sum_{a\in A_m} a^n \in A.
\]
(The inner sum vanishes for all sufficiently large $m$.)

\begin{proposition}\label{propannulation}
Assume that $n$ is not divisible by $q-1$ and  that $p$ is odd.
Then the following are equivalent:
\begin{enumerate}
\item $v_p(L(1,L/k,\psi\omega_P^n)) \geq 2/(p-1)$,
\item $\beta(n)$ and $\gamma(n)$ are divisible by $P$.
\end{enumerate}
\end{proposition}

\begin{proof}
Using the congruence
\[
	\zeta_p^i \equiv 1 + i (\zeta_p-1) \pmod{(\zeta_p-1)^2}
\]
we deduce from (\ref{psiomega}) the congruence
\[
	L(1,L/k, \psi \omega_P^n)\equiv L(1,K_P/k,\omega_P^n)+(\zeta_p-1)
	\left(\sum_{m=0}^d\sum_{a\in A_m}i(a)\omega_P(a)^n\right)
\]
modulo $(\zeta_p-1)^2$. Since $L(1,K_P/k,\omega_P^n)\in W_0$ and $p$ is odd, it follows that $L(1,L/k,\psi\omega_P^n)$ vanishes modulo $(\zeta_p-1)^2$ if and only if both
\[
	L(1,K_P/k,\omega_P^n)\equiv 0 \pmod{p}
\]
and
\[
	\sum_{m=0}^d\sum_{a\in A_m}i(a)\omega_P(a)^n \equiv 0 \pmod{\zeta_p-1}.
\]
The first congruence holds if and only if $P$ divides $\beta(n)$ and the second if and only if $P$ divides $\gamma(n)$.
\end{proof}

\subsection{Divisor class groups}

Let $E$ be a finite extension of $k$, with constant field $\bF_{q^n}$. We have an exact sequence
\[
	0 \to \bF_{q^n}^\times \to E^\times \to \Div^0 E \to \Cl E \to 0
\]
where  $\Div^0 E$ is the group of degree $0$ divisors on $E$ and $\Cl E$ the group of divisor classes of 
degree $0$ of $E$. Since $W$ is flat over $\bZ$ this sequence remains exact after tensoring with $W$, and
since $\bF_{q^n}^\times$ has no $p$-torsion we obtain
a short exact sequence
\begin{equation}\label{CEdef}
	0 \to W \otimes E^\times \to W \otimes \Div^0 E \to C(E) \to 0,
\end{equation}
where $C( E ) = W \otimes \Cl E$. 

\begin{proposition}
Let $F/E$ be a finite Galois extension with Galois group $G$. Then there is a natural short exact sequence
\[
	0 \to C(E) \to C(F)^G \to W \otimes \frac{ (\Div F)^G }{ \Div E }
\]
and $(\Div F)^G/\Div E$ is generated by the ramified primes.
\end{proposition}

\begin{proof} By Hilbert 90 we have ${\mathrm H}^1( G, F^\times )=0$, and since $W$ is flat over $\bZ$ we have
\[
	{\mathrm H}^1( G, W \otimes F^\times ) = W \otimes {\mathrm H}^1( G, F^\times )=0.
\]
Taking $G$-invariants in the sequence (\ref{CEdef}) for $F$ 
gives a short exact sequence
\[
	0 \to W\otimes E^\times \to W \otimes (\Div^0 F)^G \to C(F)^G \to 0.
\]
Comparing this with 	(\ref{CEdef}) for $E$ gives the desired exact sequence. 
\end{proof}

\begin{corollary}\label{cordivs}
Assume that $n$ is not divisible by $q-1$. Then 
\[
	C(K_P)(\omega_P^{-n}) = C(L)(\widetilde{\omega_P}^{-n})^{\Gal(L/K_P)}
\]
and
\[
	C(L)(\widetilde{\omega_P}^{-n}) = C(\widetilde{K_Q})(\omega_Q^{-n(q^{pd}-1)/(q^d-1)}).
\]
\end{corollary}

\begin{proof}
Since $L/K_P$ is unramified away from the primes above $\infty$, we have that
$ (\Div L)^G/\Div K_P$ is generated by the primes above $\infty$.   Let $S$ be the set of primes of $L$ above $\infty$ and $W^S$ the free $W$-module with basis $S$. Because $n$ is not divisible by $q-1$ we have 
\[
	W^S(\widetilde{\omega_P}^{-n})=0,
\]
hence the first claim follows from the Proposition. For the second claim, use that $\widetilde{K_Q}/L$ is unramified away from $Q$ and the primes above $\infty$, and that $Q$ is totally ramified. 
\end{proof}

\begin{theorem} \label{Theoremcyc}
Assume $p\neq 2$.  Let $P\in A$ be monic irreducible of degree $d$, and such that $i(P)\neq 0.$ 
Let $n$ be an integer such that $1\leq n \leq q^d-2$, not divisible by $q-1$ and such that $\beta(n)$ and $\gamma(n)$ are divisible by $P$.  Then $C(\widetilde{K_Q})(\omega_Q^{-n(q^{pd}-1)/(q^d-1)})$ is not cyclic.
\end{theorem}

\begin{proof}
Set 
\[
	U=C(L)(\widetilde {\omega_P}^{-n})=C(\widetilde{K_Q})(\omega_Q^{-n(q^{pd}-1)/(q^d-1)})
\]
and assume that $U$ is  $W$-cyclic, but that $\beta(n)$ and $\gamma(n)$ are divisible by $P$. 

Since $\beta(n)$ is divisible by $P$, it follows that $U^{\Gal(L/K_P)}=C(K_P)(\omega_P^{-n})$ is nonzero, and in particular that $U$ is nonzero. Let $x\in U$ be a generator, so that $U= Wx$. Let $g$ be a generator of $\Gal(L/K_P).$ We have $gx=wx$ for some $w\in W^\times$. This implies that $w^px=x$, and it follows that $w^p-1\equiv 0\pmod{p}$ and $w\equiv 1\pmod {p}.$ Since $v_p(1+w+\cdots+w^{p-1})=1$ we find
\[
	pU = (1+w+\cdots + w^{p-1})U \subset U^{\Gal(L/K_P)}
\]
and therefore the length of $U/U^{\Gal(L/K_P)}$ is at most $1$. 

On the other hand,
by \cite{GOS&SIN} and Corollary \ref{cordivs}, we have that the length of  $U/U^{\Gal(L/K_P)}$ equals
\[
	v_p(L(1,L/\tilde{k},\widetilde{\omega_P}^{n}))-v_p(L(1,K_P/k,\omega_P^n))
\]
and by (\ref{Ldecomp}) this equals
\[
	(p-1)v_p(L(1,L/k,\psi\omega_P^n)).
\]
From Proposition \ref{propannulation} we deduce that the length of $U/U^{\Gal(L/K_P)}$ is at least $2$, a contradiction.
\end{proof}

\section{Kummer-Vandiver}
If $P\in A$ is monic irreducible we write $Y_P$ for the spectrum of the integral closure of $A$ in $K_P$.

\begin{theorem}\label{thmKV}
Assume $p\not =2$. Let $P\in A$ be monic irreducible of degree $d$ and such that $i(P)\not =0.$
Let $n$ be an integer such that
\begin{enumerate}
\item $\beta(n)$ is divisible by $P$ if $n$ is not divisible by $q-1$;
\item $\gamma(n)$ is divisible by $P$.
\end{enumerate}
Let $Q(T)=P(T^p-T)$ and $N=n(q^{pd}-1)/(q^d-1)$. Then $Q$ is irreducible in $A$ and 
\[
	{\mathrm H}^1( Y_{Q,\fl}, C[Q]^{\mathrm D} )(\omega_Q^{-N-1}) \neq \{0\}.
\]
\end{theorem}

\begin{proof}
We split the proof in cases depending on the divisibility of $n$ and $dn$ by $q-1$. Note that $N$ is divisible by $q-1$ if and only if $nd$ is divisible by $q-1$.

\emph{Case 1.}
Assume that $n$ is divisible by $q-1$. This case is treated in \cite{ANG}. By \cite[Proposition 2.6]{ANG}, we get
\[
	(W\otimes\Pic Y_Q)(\omega_Q^{-N}) \neq \{0\}.
\]
Without loss of generality we may assume that $1\leq n < q^d-1$. Then by the work of Okada (\cite{OKA}, see also \cite[\S 8.20]{GOS}):
\[
	B(q^{pd}-1-N)\equiv 0\pmod{Q(T)}.
\]
By Theorem 1 of \cite{TAE1} (the ``Herbrand-Ribet theorem'') we conclude
\[
	{\mathrm H}^1( Y_{Q,\fl}, C[Q]^{\mathrm D} )(\omega^{-N-1}) \neq \{0\}.
\]	
\par

\emph{Case 2.}
Now assume that $n$ is not divisible by $q-1$ but $dn$ is. Then by Theorem \ref{Theoremcyc} the module
 $C(K_Q) ({\omega_Q}^{-N})$ is not cyclic, and so we must have:
\[
	(W\otimes \Pic Y_Q)({\omega_Q}^{-N})\not = \{ 0\}.
\]
We conclude with the same argument as in case 1.

\emph{Case 3.}
Assume that $nd$ is not divisible by $q-1$. As in the previous case, we find that 
\[
	(W\otimes \Pic Y_Q)({\omega_Q}^{-N})\not = \{ 0\}.
\]
Now we conclude we a different argument. By the above non vanishing, and by exact sequence (2) of \cite{TAE1} we find that the space of Cartier-invariant $\omega^{-N}$-typical differential forms 
\[
	A/Q \otimes_{\bF_q} \Gamma(Y,\Omega_Y)^{{\mathrm{c}}=1} ( \omega^{-N} )
\]
is at least two-dimensional. With the exact sequence of Theorem 2 of \emph{loc. cit.} one concludes that
\[
	{\mathrm H}^1( Y_{Q,\fl}, C[Q]^{\mathrm D} )(\omega^{-N-1}) \neq \{0\}.
\]	
(Using the fact that the $\omega^{-N}$-part of the last module of the exact sequence of Theorem 2 is 1-dimensional. Note that the same argument is used in the proof of Theorem 1 of \emph{loc. cit.}, see \cite[\S 4]{TAE1}.)
\end{proof}

\subsection{An example}
Let $q=3$. One can verify that with $P=T^3-T^2+1$ and $n=13$ the conditions of Theorem \ref{thmKV} are satisfied. Indeed, one has
\[
	\beta(13) = -T^9 -T^3 -T + 1 
\]
and 
\[
	\gamma(13) = -T^{12} -T^{10} + T^9 -T^4 + T^3 + T -1, 
\]
both of which are divisible by $P=T^3-T^2+1$ in $\bF_3[T]$. Also, note that $n$ is not divisible by $q-1$. Using Theorem \ref{thmKV} we thus find that the prime
\[
	Q(T) = P(T^3-T) = T^9 - T^6 - T^4 - T^3 - T^2 + 1
\]
satisfies
\[
	{\mathrm H}^1( Y_{Q,\fl}, C[Q]^{\mathrm D} )(\omega^{-N-1}) \neq \{0\}
\]
where we have
\[
	N=n\frac{q^{pd}-1}{q^d-1} = 9841.
\]
This is the counterexample to the analogue of the Kummer-Vandiver conjecture stated at the end of the introduction (we have $-9842 \equiv 9840 \pmod{q^{pd}-1}$.)

\section{Characteristic $p=2$}

We now assume that $p=2$. With some minor changes, the above arguments still work, but the result is weaker. 

We keep the notations of section \ref{seccyc}.  If $P(T)$ is a prime in $A$ of degree $d$ such that $i(P)\not =0$, then $Q(\theta)=P(T)$  is a prime of degree $2d$ in $\widetilde{A}=\mathbb F_q[\theta ],$ where $\theta^2-\theta =T$.  Set again $L=\widetilde{k}K_P\subset \widetilde{K_Q}.$ We have the following version of Theorem 
\ref{Theoremcyc}.

\begin{theorem}\label{Theoremcyc2}
Assume $p= 2$.  Let $P\in A$ be monic irreducible of degree $d$, and such that $i(P)\neq 0.$ 
Let $n$ be an integer such that $1\leq n \leq q^d-2$, not divisible by $q-1$ and such that
$\gamma(n)$ is divisible by $P$ and $L(1,K_P/k,\omega_P)$ is divisible by $4$.  Then $C(K_Q)(\omega_Q^{-n(q^{pd}-1)/(q^d-1)})$ is not cyclic.
\end{theorem}

Note that $\beta(n)$ is divisible by $P$ if and only if $L(1,K_P/k,\omega_P)$ is divisible by $2$, so the hypothesis are stronger than those in Theorem \ref{Theoremcyc}.

\begin{proof}[Proof of Theorem \ref{Theoremcyc2}]
The proof is almost identical to that of Theorem \ref{Theoremcyc}. Let $\psi$ be the unique non-trivial character of $G={\rm Gal}(\widetilde{k}/k).$

Proposition \ref{propannulation} does no longer hold, since we no longer have that $(\zeta_p-1)^2$ divides $p$. However, if $\gamma(n)$ is divisible by $P$ and if $L(1,K_P/k,\omega_P)$ is divisible by $4$ (instead of $2$), we can still conclude
\[
 	v_p( L(1,L/k,\psi\omega_P^n) ) \geq 2.
\]
Denote the length of $U=C(L)(\widetilde \omega_P^{-n})$ by $N$. We have
\[
	N = v_p(  L(1,L/k,\psi\omega_P^n) ) + v_p( L(1,L/K,\omega_P^n) ) \geq 4.
\]
Let $g$ be the nontrivial element of $\Gal( L/K_P )$. Then  the length of $U^g$ equals
$v_p( L(1,L/K,\omega_P^n) )$, which by hypothesis is at least $2$. 

Suppose that $U$ is a cyclic $W$-module, and let $x\in U$ be a generator, so that $U=Wx$. There is a $w\in W^\times$ so that $gx=wx$. We then have that $w^2-1$ is divisible by $2^N$. We find that $w-1$ is divisible by $2^{N-1}$ but not by $2^N$ (since $U^g \neq U$.) It follows that $(1+g)U = 2U$, and as in the proof of
Theorem \ref{Theoremcyc} we conclude that $U/U^g$ has length at most $1$, contradicting our hypothesis. We conclude that $U$ cannot be $W$-cyclic.
\end{proof}

Using this, we get the following variation of Theorem \ref{thmKV} in characteristic 2, with the same proof.

\begin{theorem}\label{thmKV2}
Assume $p=2$. Let $P\in A$ be monic irreducible of degree $d$ and such that $i(P)\not =0.$
Let $n$ be an integer such 
\begin{enumerate}
\item  $L(1,\omega_P^n)$ is divisible by $4$ if $n$ is not divisible by $q-1$;
\item $\gamma(n)$ is divisible by $P$.
\end{enumerate}
Let $Q(T)=P(T^2-T)$ and $N=n(q^d+1)$. Then $Q$ is irreducible in $A$ and 
\[
	{\mathrm H}^1( Y_{Q,\fl}, C[Q]^{\mathrm D} )(\omega^{-N-1}) \neq \{0\}.
\]
\end{theorem}

\subsection{An example}
Let $q=4$ and $\bF_4=\bF_2(\alpha)$. Then $P = T^5 + \alpha^2 T^4 + T^3 + \alpha T^2 + \alpha^2$ 
with $n=341=(4^5-1)/(4-1)$ satisfies the hypothesis, leading to a counterexample $Q \in \bF_4[T]$ to 
Question \ref{QKV}.

\section{Heuristics}\label{sech}

This section contains no mathematical theorems, but only crude heuristic arguments and numerical observations.
Our main goal is to convince the reader that one could \emph{a priori} expect to construct many counterexamples using the above base change strategy.

The arguments are specific to \emph{odd} $q$ so we assume throughout the section that $q$ is odd.

We argue that  one could expect that Theorem \ref{thmKV} yields at least $cX^{1/p}(\log X)^{-1}$ counter-examples of residue cardinality at most $X$ to Question \ref{QKV} (for some constant $c>0$), which is much more than the $\log \log X$ counter-examples predicted by Washington's heuristics.

In fact we will only consider counter-examples of a particular form. Assume that $q$ is odd. 
Note that by Theorem \ref{thmKV}, if we are given
\begin{enumerate}
\item an integer $m$ with $1\leq m < q-1$;
\item a monic irreducible $P \in A$ of degree $d$,
\end{enumerate}
such that
\begin{enumerate}
\item $q-1$ does not divide $md$;
\item $\beta(m(q^d-1)/(q-1))\equiv\gamma(m(q^d-1)/(q-1))\equiv 0 \pod{P}$;
\end{enumerate}
then $Q(T)=P(T^p-T)$ is monic irreducible of degree $pd$ and 
\[
	{\mathrm H}^1( Y_{Q,\fl}, C[Q]^{\mathrm D} )(\omega^{1+m(q^{pd}-1)/(q-1)}) \neq \{0\},
\]
giving a counterexample to Question \ref{QKV}.

\smallskip

The reason to restrict to $n$ of the form $m(q^d-1)/(q-1)$ lies in the following trivial observation:
\begin{lemma}\label{lem}
Let $P \in A$ be irreducible of degree $d$. If $n$ is a multiple of $(q^d-1)/(q-1)$ then  $\beta(n)$ and $\gamma(n)$ modulo $P$ lie inside $\bF_q \subset A/P$.  \qed
\end{lemma}
So we may expect that $\beta(n)$ and $\gamma(n)$ are much more likely to vanish modulo a prime $P$ of degree $d$ if $n$ is of the form $m(q^d-1)/(q-1)$.

\smallskip

Assume that $q-1$ does not divide $md$. We make the following hypotheses on a ``random'' monic irreducible $P$ of degree $d$:
\begin{enumerate}
\item $i(P)$ is non-zero with probability $(p-1)/p$;
\item $\beta(m(q^d-1)/(q-1))$ is zero modulo $P$ with probability $1/q$;
\item $\gamma(m(q^d-1)/(q-1))$ is zero modulo $P$ with probability $1/q$;
\item the above probabilities are independent of each other, and independent of the vanishing of $i(P)$.
\end{enumerate}

The first hypothesis is essentially an instance of the Chebotarev density theorem, the second and the third are motivated by Lemma \ref{lem}, and the fourth is nothing more than wishful thinking. To some extent one can verify these statements experimentally. In Table 1 we reproduce some numerical data regarding these hypotheses. Note that in the example of Table 1 $\beta$ seems to have a slight bias towards vanishing, we have no explanation for this bias.

Finally, we show that under the above hypothesis, for some $c>0$ we find that for all $X$ sufficiently large there are at least $cX^{1/p}(\log X)^{-1}$ primes of residue cardinality at most $X$ that contradict Question \ref{QKV}.

Indeed, for all $X$ sufficiently large there is a positive integer $d$ with
\[
	(\log_q X^{1/p})-2 < d \leq \log_q X^{1/p}
\]
and $d$ not divisible by $q-1$. Taking $m=1$, we should find 
\[
	\frac{p-1}{p} \cdot \frac{1}{q} \cdot \frac{1}{q} \cdot \frac{q^d}{d} 
	\geq c\frac{X^{1/p}}{\log X} 
\]
monic irreducibles $P$ of degree $d$ satisfying the conditions (with $c>0$ independent of $X$). Each of these leads to a counter-example $Q$ of residue cardinality at most $X$.

\bigskip


\begin{table}[h!]
\begin{tabular}{c|c|c|c}
 $d$ &
 $\beta(n)\equiv 0 \pod{P}$ &
 $\gamma(n)\equiv 0 \pod{P}$ & 
 $\beta(n)\equiv\gamma(n)\equiv 0 \pod{P}$ \\
\hline
$9$ & 428 & 318 & 142 \\
$11$ & 395 & 344 & 137 \\
$13$ & 416 & 332 & 147 
\end{tabular}
\smallskip\label{tab1}
\caption{The number of $P$ satisfying various congruences, out of random samples of $1000$ primes $P$ in $\bF_3[t]$ with $i(P)\neq 0$, of degrees $9$, $11$ and $13$. Again $n=(q^d-1)/2$. Note that every $P$ counted in the rightmost column gives rise to a prime $Q$ which gives a counterexample to Question \ref{QKV}.}
\label{tab2}
\end{table}


 \end{document}